\documentclass[12pt,letterpaper]{article}
\usepackage{amssymb, amsthm, amsmath}
\usepackage{tikz}
\usepackage{hyperref}
\usepackage[margin=1in]{geometry}

\def\picCCC{\begin{tabular}{c@{\qquad}c@{\qquad}c@{\qquad}c@{\qquad}c}
\begin{tikzpicture}[scale=0.75]
\draw[rounded corners, dashed, color=white!30!black]
    (-1.25,3.75) rectangle (1.25, 4.25)
    (-0.75,1.75) rectangle (0.75,2.25)
    (-1,-0.25) rectangle (1,0.25);
\node[right] at (1.25,4) {$a$};
\node[right] at (0.75,2) {$b$};
\node[right] at (1,0) {$c$};
\node[draw, circle, inner sep = 2pt] (a1) at (-1,4) {};
\node[draw, circle, inner sep = 2pt] (a2) at (-0.5,4) {};
\node[draw, circle, inner sep = 2pt] (a3) at (0,4) {};
\node[draw, circle, inner sep = 2pt] (a4) at (0.5,4) {};
\node[draw, circle, inner sep = 2pt] (a5) at (1,4) {};
\node[draw, circle, inner sep = 2pt] (u) at (0,3) {};
\node[draw, circle, inner sep = 2pt] (b1) at (-0.5,2) {};
\node[draw, circle, inner sep = 2pt] (b2) at (0,2) {};
\node[draw, circle, inner sep = 2pt] (b3) at (0.5,2) {};
\node[draw, circle, inner sep = 2pt] (v) at (0,1) {};
\node[draw, circle, inner sep = 2pt] (c1) at (-.75,0) {};
\node[draw, circle, inner sep = 2pt] (c2) at (-0.25,0) {};
\node[draw, circle, inner sep = 2pt] (c3) at (0.25,0) {};
\node[draw, circle, inner sep = 2pt] (c4) at (0.75,0) {};
\draw (a1)--(u) (a2)--(u) (a3)--(u) (a4)--(u) (a5)--(u) (u)--(b1) (u)--(b2) (u)--(b3) (b1)--(v) (b2)--(v) (b3)--(v) (v)--(c1) (v)--(c2) (v)--(c3) (v)--(c4);
\end{tikzpicture}&
\begin{tikzpicture}[scale=0.75]
\draw[rounded corners, dashed, color=white!30!black]
    (-1.25,3.75) rectangle (1.25, 4.25)
    (-0.75,1.75) rectangle (0.75,2.25)
    (-1,-0.25) rectangle (1,0.25);
\node[right] at (1.25,4) {$a$};
\node[right] at (0.75,2) {$b$};
\node[right] at (1,0) {$c$};
\node[draw, circle, inner sep = 2pt] (a1) at (-1,4) {};
\node[draw, circle, inner sep = 2pt] (a2) at (-0.5,4) {};
\node[draw, circle, inner sep = 2pt] (a3) at (0,4) {};
\node[draw, circle, inner sep = 2pt] (a4) at (0.5,4) {};
\node[draw, circle, inner sep = 2pt] (a5) at (1,4) {};
\node[draw, circle, inner sep = 2pt] (u) at (0,3) {};
\node[draw, circle, inner sep = 2pt] (b1) at (-0.5,2) {};
\node[draw, circle, inner sep = 2pt] (b2) at (0,2) {};
\node[draw, circle, inner sep = 2pt] (b3) at (0.5,2) {};
\node[draw, circle, inner sep = 2pt] (v) at (0,1) {};
\node[draw, circle, inner sep = 2pt] (c1) at (-.75,0) {};
\node[draw, circle, inner sep = 2pt] (c2) at (-0.25,0) {};
\node[draw, circle, inner sep = 2pt] (c3) at (0.25,0) {};
\node[draw, circle, inner sep = 2pt] (c4) at (0.75,0) {};
\draw (a1)--(u) (a2)--(u) (a3)--(u) (a4)--(u) (a5)--(u) (u)--(b1) (u)--(b2) (u)--(b3) (b1)--(v) (b2)--(v) (b3)--(v) (v)--(c1) (v)--(c2) (v)--(c3) (v)--(c4);
\draw[rounded corners] (u)--(-1,3)--(-1,1)--(v);
\end{tikzpicture}&
\begin{tikzpicture}[scale=0.75]
\draw[rounded corners, dashed, color=white!30!black]
    (-1,0.75) rectangle (1,1.25);
\node[right] at (1,1) {$a$};
\node[draw, circle, inner sep = 2pt] (a1) at (-.75,1) {};
\node[draw, circle, inner sep = 2pt] (a2) at (-0.25,1) {};
\node[draw, circle, inner sep = 2pt] (a3) at (0.25,1) {};
\node[draw, circle, inner sep = 2pt] (a4) at (0.75,1) {};
\node[draw, circle, inner sep = 2pt] (u) at (0,0) {};
\node[draw, circle, inner sep = 2pt] (v) at (0,-1) {};
\node[draw, circle, inner sep = 2pt] (w) at (-0.5,-2) {};
\node[draw, circle, inner sep = 2pt] (x) at (0.5,-2) {};

\draw (a1)--(u) (a2)--(u) (a3)--(u) (a4)--(u) (u)--(v)--(w)--(x)--(v);
\end{tikzpicture}&
\begin{tikzpicture}[scale=0.75]
\draw[rounded corners, dashed, color=white!30!black]
    (-1,0.75) rectangle (1,1.25);
\node[right] at (1,1) {$a$};
\node[draw, circle, inner sep = 2pt] (a1) at (-.75,1) {};
\node[draw, circle, inner sep = 2pt] (a2) at (-0.25,1) {};
\node[draw, circle, inner sep = 2pt] (a3) at (0.25,1) {};
\node[draw, circle, inner sep = 2pt] (a4) at (0.75,1) {};
\node[draw, circle, inner sep = 2pt] (u) at (0,0) {};
\node[draw, circle, inner sep = 2pt] (v) at (-0.5,-1) {};
\node[draw, circle, inner sep = 2pt] (w) at (0.5,-1) {};
\node[draw, circle, inner sep = 2pt] (x) at (0,-2) {};

\draw (a1)--(u) (a2)--(u) (a3)--(u) (a4)--(u) (u)--(v)--(w)--(x)--(v) (u)--(w);
\end{tikzpicture}&
\begin{tikzpicture}[scale=0.75]
\draw[rounded corners, dashed, color=white!30!black]
    (-1,0.75) rectangle (1,1.25);
\node[right] at (1,1) {$a$};
\node[draw, circle, inner sep = 2pt] (a1) at (-.75,1) {};
\node[draw, circle, inner sep = 2pt] (a2) at (-0.25,1) {};
\node[draw, circle, inner sep = 2pt] (a3) at (0.25,1) {};
\node[draw, circle, inner sep = 2pt] (a4) at (0.75,1) {};
\node[draw, circle, inner sep = 2pt] (u) at (0,0) {};
\node[draw, circle, inner sep = 2pt] (v) at (0,{-1-sqrt(3)/8}) {};
\node[draw, circle, inner sep = 2pt] (w) at (-1,-2) {};
\node[draw, circle, inner sep = 2pt] (x) at (1,-2) {};

\draw (a1)--(u) (a2)--(u) (a3)--(u) (a4)--(u) (u)--(v)--(w)--(x)--(v) (w)--(u)--(x);
\end{tikzpicture}\\
(a) $GS_{a,b,c}$&
(b) $GS_{a,b,c}^{+}$&
(c) $\text{Paw}_a$&
(d) $\text{Di}_a$&
(e) $(K_4)_a$
\end{tabular}}

\def\MisaA{\begin{tikzpicture}[scale=0.75]
\draw[dashed,color=white!30!black] (-1.5,2)--(1.5,2);

\node[draw, circle, inner sep = 2pt] (a1) at (-1.4,4) {};
\node[draw, circle, inner sep = 2pt] (a2) at (-1,4) {};
\node[draw, circle, inner sep = 2pt] (a3) at (-0.6,4) {};
\node[draw, circle, inner sep = 2pt] (a4) at (-0.2,4) {};
\node[draw, circle, inner sep = 2pt] (a5) at (0.2,4) {};
\node[draw, circle, inner sep = 2pt] (a6) at (0.6,4) {};
\node[draw, circle, inner sep = 2pt] (a7) at (1,4) {};
\node[draw, circle, inner sep = 2pt] (a8) at (1.4,4) {}
;
\node[draw, circle, inner sep = 2pt] (u) at (0,3) {};
\node[draw, circle, inner sep = 2pt,fill=white] (b) at (0,2) {};
\node[draw, circle, inner sep = 2pt] (v) at (0,1) {};
\node[draw, circle, inner sep = 2pt] (c) at (0,0) {};
\draw (a1)--(u) (a2)--(u) (a3)--(u) (a4)--(u) (a5)--(u) (a6)--(u) (a7)--(u) (a8)--(u) (u)--(b)  (b)--(v) (v)--(c) ;
\end{tikzpicture}
\hfil
\begin{tikzpicture}[scale=0.75]

\draw[dashed,color=white!30!black] (-1.5,2)--(1.5,2);
\node[draw, circle, inner sep = 2pt] (a) at (0,4) {};
\node[draw, circle, inner sep = 2pt] (u) at (0,3) {};
\node[draw, circle, inner sep = 2pt,fill=white] (b1) at (-1.05,2) {};
\node[draw, circle, inner sep = 2pt,fill=white] (b2) at (-.75,2) {};
\node[draw, circle, inner sep = 2pt,fill=white] (b3) at (-0.45,2) {};
\node[draw, circle, inner sep = 2pt,fill=white] (b4) at (-0.15,2) {};
\node[draw, circle, inner sep = 2pt,fill=white] (b5) at (0.15,2) {};
\node[draw, circle, inner sep = 2pt,fill=white] (b6) at (0.45,2) {};
\node[draw, circle, inner sep = 2pt,fill=white] (b7) at (0.75,2) {};
\node[draw, circle, inner sep = 2pt,fill=white] (b8) at (1.05,2) {}
;
\node[draw, circle, inner sep = 2pt] (v) at (0,1) {};
\node[draw, circle, inner sep = 2pt] (c) at (0,0) {};
\draw (a)--(u) (u)--(b1) (u)--(b2) (u)--(b3) (u)--(b4) (u)--(b5) (u)--(b5) (u)--(b6) (u)--(b7) (b8)--(u) (b1)--(v) (b2)--(v) (b3)--(v) (b4)--(v) (b5)--(v) (b6)--(v) (b7)--(v) (b8)--(v) (v)--(c);
\end{tikzpicture}
\hfil
\begin{tikzpicture}[scale=0.75]
\draw[dashed,color=white!30!black] (-1.5,2)--(1.5,2);
\node[draw, circle, inner sep = 2pt] (a1) at (-1,4) {};
\node[draw, circle, inner sep = 2pt] (a2) at (-0.5,4) {};
\node[draw, circle, inner sep = 2pt] (a3) at (0,4) {};
\node[draw, circle, inner sep = 2pt] (a4) at (0.5,4) {};
\node[draw, circle, inner sep = 2pt] (a5) at (1,4) {};
\node[draw, circle, inner sep = 2pt] (u) at (0,3) {};
\node[draw, circle, inner sep = 2pt,fill=white] (b1) at (-0.75,2) {};
\node[draw, circle, inner sep = 2pt,fill=white] (b2) at (-0.25,2) {};
\node[draw, circle, inner sep = 2pt,fill=white] (b3) at (0.25,2) {};
\node[draw, circle, inner sep = 2pt,fill=white] (b4) at (0.75,2) {};
\node[draw, circle, inner sep = 2pt] (v) at (0,1) {};
\node[draw, circle, inner sep = 2pt] (c) at (0,0) {};
\draw (a1)--(u) (a2)--(u) (a3)--(u) (a4)--(u) (a5)--(u)  (u)--(b1) (u)--(b2) (u)--(b3) (u)--(b4) (b1)--(v) (b2)--(v) (b3)--(v) (b4)--(v) (v)--(c);
\end{tikzpicture}}

\def\FutureA{\begin{tikzpicture}[scale=1]
\node[draw,circle,inner sep= 2pt] (a0) at (0,0) {};
\node[draw,circle,inner sep= 2pt] (a1) at (30:1) {};
\node[draw,circle,inner sep= 2pt] (a2) at (90:1) {};
\node[draw,circle,inner sep= 2pt] (a3) at (150:1) {};
\node[draw,circle,inner sep= 2pt] (a4) at (210:1) {};
\node[draw,circle,inner sep= 2pt] (a5) at (270:1) {};
\node[draw,circle,inner sep= 2pt] (a6) at (330:1) {};
\node[draw,circle,inner sep= 2pt] (a7) at (30:2) {};
\node[draw,circle,inner sep= 2pt] (a8) at (150:2) {};
\node[draw,circle,inner sep= 2pt] (a9) at (270:2) {};

\draw (a3)--(a0)--(a1)--(a2)--(a3)--(a4)--(a5)--(a6)--(a1)--(a7) (a3)--(a8)  (a5)--(a9);
\end{tikzpicture}
}

\def\FutureB{\begin{tikzpicture}[scale=0.75]
    \node[draw,circle,inner sep= 2pt] (a0) at (-1,0) {};
    \node[draw,circle,inner sep= 2pt] (a1) at (1,0) {};
    \node[draw,circle,inner sep= 2pt] (b0) at (0,-1) {};
    \node[draw,circle,inner sep= 2pt] (c0) at (-1,-2) {};
    \node[draw,circle,inner sep= 2pt] (c1) at (0,-2) {};
    \node[draw,circle,inner sep= 2pt] (c2) at (1,-2) {};
    \node[draw,circle,inner sep= 2pt] (d0) at (0,-3) {};
    \node[draw,circle,inner sep= 2pt] (e0) at (0,-4) {};
    \node[draw,circle,inner sep= 2pt] (f0) at (0,-5) {};
    \node[draw,circle,inner sep= 2pt] (g0) at (0,-6) {};

\draw
(a0)--(b0)--(a1) (b0)--(c0)--(d0)--(c1)--(b0)--(c2)--(d0)--(e0)--(f0)--(g0);
\end{tikzpicture}}

\def\FutureC{\begin{tikzpicture}[scale=1]
\node[draw,circle,inner sep= 2pt] (a0) at (0,0) {};
\node[draw,circle,inner sep= 2pt] (a1) at (30:1) {};
\node[draw,circle,inner sep= 2pt] (a2) at (90:1) {};
\node[draw,circle,inner sep= 2pt] (b1) at (90:0.3333) {};
\node[draw,circle,inner sep= 2pt] (b2) at (90:0.6667) {};
\node[draw,circle,inner sep= 2pt] (a3) at (150:1) {};
\node[draw,circle,inner sep= 2pt] (a4) at (210:1) {};
\node[draw,circle,inner sep= 2pt] (a5) at (270:1) {};
\node[draw,circle,inner sep= 2pt] (a6) at (330:1) {};
\node[draw,circle,inner sep= 2pt] (a7) at (30:2) {};
\node[draw,circle,inner sep= 2pt] (a8) at (150:2) {};
\node[draw,circle,inner sep= 2pt] (c0) at (265:2) {};
\node[draw,circle,inner sep= 2pt] (c1) at (275:2) {};
\node[draw,circle,inner sep= 2pt] (c2) at (255:2) {};
\node[draw,circle,inner sep= 2pt] (c3) at (285:2) {};

\draw (a3)--(a0)--(a1)--(a2)--(a3)--(a4)--(a5)--(a6)--(a1)--(a7) (a3)--(a8)  (a5)--(c0) (a5)--(c1) (a5)--(c2) (a5)--(c3)
(a3)--(b1)--(a1)
(a3)--(b2)--(a1);
\end{tikzpicture}}

\def\FutureD{\begin{tikzpicture}[scale=0.75]
    \node[draw,circle,inner sep= 2pt] (a0) at (-1,0) {};
    \node[draw,circle,inner sep= 2pt] (a1) at (1,0) {};
    \node[draw,circle,inner sep= 2pt] (a2) at (0,0) {};
    \node[draw,circle,inner sep= 2pt] (b0) at (0,-1) {};
    \node[draw,circle,inner sep= 2pt] (c0) at (-1,-2) {};
    \node[draw,circle,inner sep= 2pt] (c1) at (0,-2) {};
    \node[draw,circle,inner sep= 2pt] (c2) at (1,-2) {};
    \node[draw,circle,inner sep= 2pt] (d0) at (0,-3) {};
    \node[draw,circle,inner sep= 2pt] (e0) at (-1,-4) {};
    \node[draw,circle,inner sep= 2pt] (e1) at (0,-4) {};
    \node[draw,circle,inner sep= 2pt] (e2) at (1,-4) {};
    \node[draw,circle,inner sep= 2pt] (f0) at (0,-5) {};
    \node[draw,circle,inner sep= 2pt] (g0) at (-1,-6) {};
    \node[draw,circle,inner sep= 2pt] (g1) at (0,-6) {};
    \node[draw,circle,inner sep= 2pt] (g2) at (1,-6) {};

\draw
(a0)--(b0)--(a1) (b0)--(c0)--(d0)--(c1)--(b0)--(c2)--(d0)--(e0)--(f0)--(g0)
(a2)--(b0)
(e1)--(d0)--(e2)
(e1)--(f0)--(g1) (e2)--(f0)--(g2);
\end{tikzpicture}}

\def\picCONDENSED{\begin{tabular}{c@{\qquad\qquad}c@{\qquad\qquad}c}
\begin{tikzpicture}[scale=1]
\node[draw, circle, inner sep = 2pt] (a) at (0,1) {};
\node[draw, circle, inner sep = 2pt] (u) at (0,0) {};
\node[draw, circle, inner sep = 2pt] (v) at (0,-1) {};
\node[draw, circle, inner sep = 2pt] (w) at (-0.5,-2) {};
\node[draw, circle, inner sep = 2pt] (x) at (0.5,-2) {};

\draw (a)--(u) (u)--(v)--(w)--(x)--(v);

\node[fill=white,inner sep=1.5pt] at (0,0.5) {$a$};
\node[fill=white,inner sep=1.5pt] at (0,-0.5) {$b$};
\node[fill=white,inner sep=1.5pt] at (-0.25,-1.5) {$c$};
\node[fill=white,inner sep=1.5pt] at (0.25,-1.5) {$d$};
\node[fill=white,inner sep=1.5pt] at (0,-2) {$e$};
\end{tikzpicture}&
\begin{tikzpicture}[scale=1]
\node[draw, circle, inner sep = 2pt] (a) at (0,1) {};
\node[draw, circle, inner sep = 2pt] (u) at (0,0) {};
\node[draw, circle, inner sep = 2pt] (v) at (-0.5,-1) {};
\node[draw, circle, inner sep = 2pt] (w) at (0.5,-1) {};
\node[draw, circle, inner sep = 2pt] (x) at (0,-2) {};

\draw (a)--(u)  (u)--(v)--(w)--(x)--(v) (u)--(w);

\node[fill=white,inner sep=1.5pt] at (0,0.5) {$a$};
\node[fill=white,inner sep=1.5pt] at (-0.25,-0.5) {$b$};
\node[fill=white,inner sep=1.5pt] at (0.25,-0.5) {$c$};
\node[fill=white,inner sep=1.5pt] at (0,-1) {$d$};
\node[fill=white,inner sep=1.5pt] at (-0.25,-1.5) {$e$};
\node[fill=white,inner sep=1.5pt] at (0.25,-1.5) {$f$};

\end{tikzpicture}&
\begin{tikzpicture}[scale=1]
\node[draw, circle, inner sep = 2pt] (a) at (0,1) {};
\node[draw, circle, inner sep = 2pt] (u) at (0,0) {};
\node[draw, circle, inner sep = 2pt] (v) at (0,{-1-sqrt(3)/8}) {};
\node[draw, circle, inner sep = 2pt] (w) at (-1,-2) {};
\node[draw, circle, inner sep = 2pt] (x) at (1,-2) {};

\draw (a)--(u)  (u)--(v)--(w)--(x)--(v) (w)--(u)--(x);

\node[fill=white,inner sep=1.5pt] at (0,0.5) {$a$};
\node[fill=white,inner sep=1.5pt] at (0,-0.6) {$d$};
\node[fill=white,inner sep=1.5pt] at (0,-2) {$g$};
\node[fill=white,inner sep=1.5pt] at (-0.5,-1) {$b$};
\node[fill=white,inner sep=1.5pt] at (0.5,-1) {$c$};
\node[fill=white,inner sep=1.5pt] at (-0.5,-1.6) {$e$};
\node[fill=white,inner sep=1.5pt] at (0.5,-1.6) {$f$};

\end{tikzpicture}\\
$\text{Paw}_a$&
$\text{Di}_a$&
$(K_4)_a$
\end{tabular}}

\usepackage{tikz}
\tikzset{every picture/.style={thick}}

\newtheorem{theorem}{Theorem}
\newtheorem{proposition}{Proposition}

\newtheorem{observation}{Observation}
\theoremstyle{definition}

\allowdisplaybreaks

\title{Graphs with at most two trees\\ in a forest building process}
\author{Steve Butler\footnote{Department of Mathematics, Iowa State University, Ames, IA 50011 USA\newline {\tt\{butler,hamanaka,hardtme\}@iastate.edu}} \and Misa Hamanaka\footnotemark[1] \and Marie Hardt\footnotemark[1]}
\date{\empty}

\begin{document}
\maketitle

\begin{abstract}
Given a graph, we can form a spanning forest by first sorting the edges in some order, and then only keep edges incident to a vertex which is not incident to any previous edge.  The resulting forest is dependent on the ordering of the edges, and so we can ask, for example, how likely is it for the process to produce a graph with $k$ trees.

We look at all graphs which can produce at most two trees in this process and determine the probabilities of having either one or two trees.  From this we construct infinite families of graphs which are non-isomorphic but produce the same probabilities.
\end{abstract}

\section{Introduction}
We consider the following \emph{forest building process}:
\begin{enumerate}
\item Take all of the edges of the graph, remove them and sort them in some order.
\item Go through the edges and only put those edges back in which connects to some vertex not previously seen by any edge.  
\end{enumerate}
From this, we must end up with a forest (or graph without cycles) since we can never add an edge that closes a cycle. As an example, in Figure~\ref{fig:paw} we list all $24$ different ways to order the edges and group them based on the resulting forest formed.

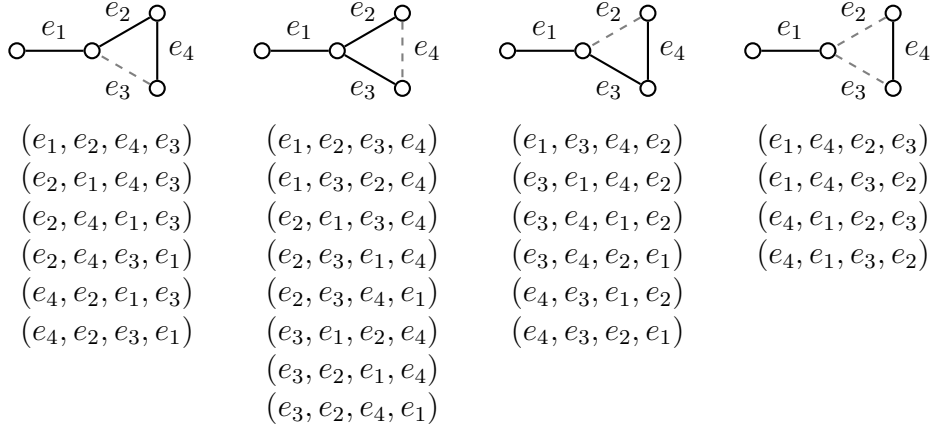
\begin{figure}[h!]
\centering

\hfil\begin{tabular}{c@{\hspace{0.25in}}c@{\hspace{0.25in}}c@{\hspace{0.25in}}c}
\begin{tikzpicture}[scale=1]
\node[draw,circle,inner sep= 2pt] (a1) at (-1,0) {};
\node[draw,circle,inner sep= 2pt] (a2) at (0,0) {};
\node[draw,circle,inner sep= 2pt] (a3) at ({sqrt(3)/2},0.5) {};
\node[draw,circle,inner sep= 2pt] (a4) at ({sqrt(3)/2},-0.5) {};
\draw[dashed,color=white!50!black] (a4)--(a2);
\draw (a1)--(a2)--(a3)--(a4);
\node [above] at (-0.5,0) {$e_1$};
\node [above] at ({sqrt(3)/5},0.25) {$e_2$};
\node [below] at ({sqrt(3)/5},-0.25) {$e_3$};
\node [right] at ({sqrt(3)/2},0) {$e_4$};
\end{tikzpicture}&
\begin{tikzpicture}[scale=1]
\node[draw,circle,inner sep= 2pt] (a1) at (-1,0) {};
\node[draw,circle,inner sep= 2pt] (a2) at (0,0) {};
\node[draw,circle,inner sep= 2pt] (a3) at ({sqrt(3)/2},0.5) {};
\node[draw,circle,inner sep= 2pt] (a4) at ({sqrt(3)/2},-0.5) {};
\draw (a1)--(a2)--(a3) (a4)--(a2);
\draw[dashed,color=white!50!black] (a3)--(a4);
\node [above] at (-0.5,0) {$e_1$};
\node [above] at ({sqrt(3)/5},0.25) {$e_2$};
\node [below] at ({sqrt(3)/5},-0.25) {$e_3$};
\node [right] at ({sqrt(3)/2},0) {$e_4$};
\end{tikzpicture}&
\begin{tikzpicture}[scale=1]
\node[draw,circle,inner sep= 2pt] (a1) at (-1,0) {};
\node[draw,circle,inner sep= 2pt] (a2) at (0,0) {};
\node[draw,circle,inner sep= 2pt] (a3) at ({sqrt(3)/2},0.5) {};
\node[draw,circle,inner sep= 2pt] (a4) at ({sqrt(3)/2},-0.5) {};
\draw (a1)--(a2) (a3)--(a4)--(a2);
\draw[dashed,color=white!50!black] (a3)--(a2);
\node [above] at (-0.5,0) {$e_1$};
\node [above] at ({sqrt(3)/5},0.25) {$e_2$};
\node [below] at ({sqrt(3)/5},-0.25) {$e_3$};
\node [right] at ({sqrt(3)/2},0) {$e_4$};
\end{tikzpicture}&
\begin{tikzpicture}[scale=1]
\node[draw,circle,inner sep= 2pt] (a1) at (-1,0) {};
\node[draw,circle,inner sep= 2pt] (a2) at (0,0) {};
\node[draw,circle,inner sep= 2pt] (a3) at ({sqrt(3)/2},0.5) {};
\node[draw,circle,inner sep= 2pt] (a4) at ({sqrt(3)/2},-0.5) {};
\draw (a1)--(a2) (a3)--(a4);
\draw[dashed,color=white!50!black] (a3)--(a2)--(a4);

\node [above] at (-0.5,0) {$e_1$};
\node [above] at ({sqrt(3)/5},0.25) {$e_2$};
\node [below] at ({sqrt(3)/5},-0.25) {$e_3$};
\node [right] at ({sqrt(3)/2},0) {$e_4$};
\end{tikzpicture}\\
$(e_1,e_2,e_4,e_3)$&$(e_1,e_2,e_3,e_4)$&$(e_1,e_3,e_4,e_2)$&$(e_1,e_4,e_2,e_3)$\\
$(e_2,e_1,e_4,e_3)$&$(e_1,e_3,e_2,e_4)$&$(e_3,e_1,e_4,e_2)$&$(e_1,e_4,e_3,e_2)$\\
$(e_2,e_4,e_1,e_3)$&$(e_2,e_1,e_3,e_4)$&$(e_3,e_4,e_1,e_2)$&$(e_4,e_1,e_2,e_3)$\\
$(e_2,e_4,e_3,e_1)$&$(e_2,e_3,e_1,e_4)$&$(e_3,e_4,e_2,e_1)$&$(e_4,e_1,e_3,e_2)$\\
$(e_4,e_2,e_1,e_3)$&$(e_2,e_3,e_4,e_1)$&$(e_4,e_3,e_1,e_2)$&\\
$(e_4,e_2,e_3,e_1)$&$(e_3,e_1,e_2,e_4)$&$(e_4,e_3,e_2,e_1)$&\\
&$(e_3,e_2,e_1,e_4)$&&\\
&$(e_3,e_2,e_4,e_1)$&&
\end{tabular}
\caption{The results from different edge orderings of the paw graph.}
\label{fig:paw}
\end{figure}

We will consider the problem:  How many different edge orderings produce a given number, say $k$, of trees in the resulting graph. Equivalently, what is the probability that if we take a random ordering of the edges,  we produce a forest with $k$ trees.  We will let $P(G,k)$ denote this probability.  From Figure~\ref{fig:paw}, we see that $P(G,1)=\frac56$ and $P(G,2)=\frac16$  (note that the probabilities need to sum to $1$).

This process was implicitly used in a paper of Butler et al.\ \cite{BCCG} for the complete graph, and explicitly introduced in a paper by Berikkyzy et al.\ \cite{B} where some basic properties were established and the probabilities for complete bipartite graphs were determined.  We summarize these results here.

\begin{theorem}[Butler et al.\ \cite{BCCG}]
We have
$\displaystyle P(K_n,k)=\frac{{n-1\choose n-2k,k,k-1}2^{n-2k}}{{2n-2\choose n}}$.
\end{theorem}

\begin{theorem}[Berikkyzy et al.\ \cite{B}]
We have $\displaystyle P(K_{s,t},k)=\frac{(s+t){s\choose k}{t\choose k}}{st{s+t\choose s}}$.
\end{theorem}

For small graphs (at most $5$ vertices), the probabilities are given in Berikkyzy et al.\ \cite{B}.  There are a few instances where two graphs would have the same probabilities for all $k$ listed, and most of those were edge-transitive graphs.  More generally, the following was observed.

\begin{observation}[Berikkyzy et al.\ \cite{B}]
If $G$ is an edge-transitive graph with minimum degree of at least $2$ and $e$ is any edge, then we have $P(G,k)=P(G-e,k)$ for all $k$.
\end{observation}

In essence, this follows by noting that the last edge in an ordering is never kept, \emph{and} by symmetry every edge is the last edge in an ordering equally often.

The goal of this note is to compute the probabilities for more families of graphs, namely graphs which can produce at most two trees in the forest building process.  Using this, we will produce infinitely many examples of non-isomorphic graphs $G$ and $H$ where the probabilities agree and neither $G$ or $H$ are edge-transitive.

\section{Graphs with at most two trees}
We are interested in exploring the graphs which can produce at most two trees in the forest building process.  Equivalently, this states that there are at most two disjoint edges in the graph (disjoint in the sense that they share no vertex).

\begin{proposition}\label{prop1}
The only non-empty graphs without isolated vertices, which contain no pair of disjoint edges, are star graphs ($K_{1,n}$) and the triangle graph ($K_3$).
\end{proposition}
\begin{proof}
If the graph is not connected, then taking one edge from two different components gives two disjoint edges.  So we may assume the graph is connected.

If the graph has two disjoint edges, then we can connect these together by a path, creating a path of length at least four.  Conversely, if the graph has a path of length at least four, then it must contain two disjoint edges.  So we can conclude the longest path is a path with at most three vertices.  If the longest path has two vertices, then the graph is a $K_2$.  

If the longest path has three vertices and the ends of the path are not leafs then it must be that the ends connect and form a triangle.  Since the paw graph has two disjoint edges, this can only happen if the graph is a $K_3$.

Finally, if we are not a triangle and don't have a path of length four (and hence no cycles), then it must be that we are a star.
\end{proof}

\begin{proposition}\label{prop:delete}
If a graph without isolated vertices has a vertex $v$ of degree at least five and contains no set of three disjoint edges, then deleting $v$ and all incident edges, and removing any  isolated vertices results in either an empty graph, a star, or a $K_3$.
\end{proposition}

\begin{proof}
Let $v$ be a vertex of degree at least $5$ in the graph.  Suppose that the resulting graph from deleting the vertex $v$ and all incident edges, and removing any isolated vertices is not the empty graph, a star, or a $K_3$.  Proposition~\ref{prop1} states that the only non-empty graphs without isolated vertices, that contain no set of two disjoint edges, is the star graph and $K_3$. Thus the resulting graph will contain at least two disjoint edges. Call these edges $e_1$ and $e_2$; note neither of these edges are incident to $v$.

At most four edges incident to $v$ are also incident to the edges $e_1$ and $e_2$.  Since $v$ has degree at least $5$, this leaves at least one edge, $e_3$, that is connected to $v$ and not incident to $e_1$ or $e_2$.  Thus the original graph contains a set of three disjoint edges.
\end{proof}

Finally, we observe that if all the degrees are bounded and the graph is connected, then as $n$ gets large, the diameter must also grow---which forces three disjoint edges.

Putting this all together, we see that for $n$ large enough (in fact, $n\ge 6$)  the only connected graphs which produce at most two trees, and are not stars, are the following five families.
\begin{itemize}
\item $GS_{a,b,c}$ -- The stars $K_{1,a+b}$ and $K_{1,b+c}$ which have $b$ leaves glued together.  (Glued stars.)
\item $GS_{a,b,c}^{+}$ -- The stars $K_{1,a+b}$ and $K_{1,b+c}$ which have $b$ leaves glued together and the centers joined by an edge.  (Glued stars with an edge.)
\item $\text{Paw}_a$ -- The paw graph with $a$ leaves appended to the vertex of degree $1$.
\item $\text{Di}_a$ -- The diamond graph (a four-cycle with an extra edge) with $a$ leaves appended to one of the vertices of degree $2$.
\item  $(K_4)_a$ -- The complete graph on four vertices with $a$ leaves appended to one of the vertices.
\end{itemize}

Note that the first two of these correspond to Proposition~\ref{prop:delete} where the remaining graph is a star; and the last three of these correspond to Proposition~\ref{prop:delete} where the remaining graph is a $K_3$.  These graphs are shown in Figure~\ref{fig:families}.

\begin{figure}[h!]
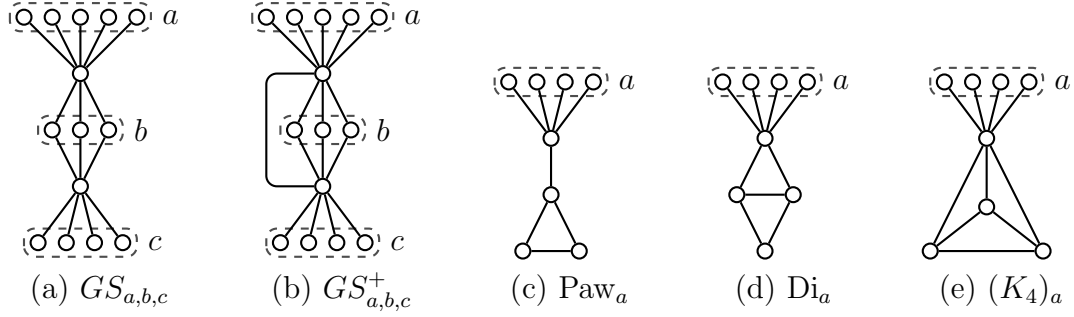

\centering

\picCCC

\caption{The five families of graphs.}
\label{fig:families}
\end{figure}

\section{Computing probabilities for the families}

We now turn our attention to computing the probabilities that a graph ends in one or two trees in the forest building process.  We can find these probabilities by noting that if there are $m$ edges in the graph, then the probability that we end with two trees is
\[
P(G,2)=\frac{\big|\{\text{rearrangements with two trees}\}\big|}{m!}.
\]
We will focus on counting the rearrangements which produce two trees.  Particularly,  we want to count rearrangements where at some point an edge occurs, not at the start, and involves two vertices which have not been previously seen.

Since we will be counting rearrangements, we will find it useful to know how to manipulate binomial coefficients.  Recall that ${n\choose k}=\frac{n!}{k!(n-k)!}$ is the number of ways to choose $k$ elements (in our case this will usually be locations) out of an $n$ element set.  There are many binomial coefficient identities (see Graham, Knuth, and Patashnik \cite[Ch.\ 5]{concrete} for a good introduction); we will need to make repeated use of the following well-known result.

\begin{proposition}\label{prop:binom}
We have 
\begin{equation}\label{eq:binom}
\sum_j{\ell-j\choose m}{q+j\choose n}={\ell+q+1\choose m+n+1}
\end{equation}
where the sum ranges over all values where the summands are nonzero.
\end{proposition}

\begin{theorem}\label{thm:GS}
We have the following probabilities.
\begin{align*}
P(GS_{a,b,c},1)&=\frac{b}{(b+c+1)(b+c)}+\frac{b}{(a+b+1)(a+b)}\\[5pt]
P(GS_{a,b,c}^{+},1)&=\frac{2b+c+2}{(b+c+1)(b+c+2)}+\frac{2b+a+2}{(b+a+1)(b+a+2)}-\frac{1}{a+2b+c+1}
\end{align*}
\end{theorem}
\begin{proof}
Since $P(GS_{a,b,c},1)+P(GS_{a,b,c},2)=1$, we can focus on computing the probability of resulting in two trees.  We now claim
\begin{multline}
P(GS_{a,b,c},2)=\sum_{i=0}^{b} \sum_{j=0}^{a} \frac {{a \choose j}{b \choose i}{(i+j)!}{(b+c-i)}{(a+2b+c-i-j-1)!}}{(a+2b+c)!} -\frac{b+c}{a+2b+c}\\
+\sum_{i=0}^{b} \sum_{j=0}^{c} \frac{{c \choose j}{b \choose i}{(i+j)!}{(a+b-i)}{(a+2b+c-i-j-1)!}}{(a+2b+c)!} -\frac{a+b}{a+2b+c}.\label{eq:GS}
\end{multline}

This comes from the two cases, namely where our first edge initially comes from the ``top half'' (i.e., edges coming from the star $K_{1,a+b}$), and where our first edge initially comes from the ``bottom half'' (i.e., edges coming from the star $K_{1,b+c}$).  We focus on the top half case, as the bottom half follows by an identical argument by interchanging the roles of $a$ and $c$.

Determining if we have two trees comes down to what happens when we pick our first edge from the star $K_{1,b+c}$.  We look at all ways that this occurs by first picking edges from $K_{1,a+b}$ and then considering what happens when we pick our edge from $K_{1,b+c}$.  In particular, we will pick $j$ edges from the $a$ leaf vertices and $i$ edges from the $b$ gluing vertices.  We now run over all possibilities for $i$ and $j$.

For each choice of edges we now consider all possible ordering as follows:
\begin{itemize}
\item ${a\choose j}$ corresponds to which of the $j$ edges among the $a$ were chosen.
\item ${b\choose i}$ corresponds to which of the $i$ edges among the $b$ were chosen.
\item $(i+j)!$ indicates how many ways to order these $i+j$ edges (note that these $i+j$ edges are all of the initial edges).
\item $(b+c-i)$ indicates how many edges disjoint from the ones above are available to choose, if we want to create two trees.
\item $(a+2b+c-i-j-1)!$ is the number of ways to rearrange the remaining edges.
\end{itemize}
This gives all orderings of edges possible, to get the probability we now divide by the total number of orderings which is $(a+2b+c)!$.

Note that in the summation we need to correct for $i=0$, $j=0$ which does not fall into the case where the first edge is from the top.  So we subtract this term off at the end, which gives the $-\frac{b+c}{a+2b+c}$ term at the end.

To now simplify these sums we can repeatedly apply \eqref{eq:binom}.  So we have the following.
\begin{align*}
&\sum_{i=0}^{b} \sum_{j=0}^{a} \frac {{a \choose j}{b \choose i}{(i+j)!}{(b+c-i)}{(a+2b+c-i-j-1)!}}{(a+2b+c)!}\\
=&\sum_{i=0}^{b} \sum_{j=0}^{a} \frac {{a! \over j!(a-j)!}{b! \over i!(b-i)!}{(i+j)!}{(b+c-i)}{(a+2b+c-i-j-1)!}}{(a+2b+c)!}\\
=&\sum_{i=0}^{b}\frac{a!b!(b+c-i)}{(a+2b+c)!(b-i)!} \sum_{j=0}^{a} \frac{(i+j)!}{i!\,j!}\frac{(a+2b+c-i-j-1)!}{(a-j)!}\\
=&\sum_{i=0}^{b}\frac{a!b!(b+c-i)(2b+c-i-1)!}{(a+2b+c)!(b-i)!} \sum_{j=0}^{a} \frac{(i+j)!}{i!\,j!}\frac{(a+2b+c-i-j-1)!}{(a-j)!(2b+c-i-1)!}\\
=&\sum_{i=0}^{b}\frac{a!b!(b+c-i)(2b+c-i-1)!}{(a+2b+c)!(b-i)!} \sum_{j=0}^{a} {i+j\choose i}{a+2b+c-i-1-j\choose 2b+c-i-1}\\
=&\sum_{i=0}^{b}\frac{a!b!(b+c-i)(2b+c-i-1)!}{(a+2b+c)!(b-i)!} {a+2b+c\choose 2b+c}\\
=&\sum_{i=0}^{b}\frac{a!b!(b+c-i)(2b+c-i-1)!}{(a+2b+c)!(b-i)!} \frac{(a+2b+c)!}{(2b+c)!a!}\\
=&\sum_{i=0}^{b}\frac{b!(b+c-i)(2b+c-i-1)!}{(b-i)!(2b+c)!}\\
=&\frac{b!(b+c-1)!}{(2b+c)!}\sum_{i=0}^{b}\frac{(2b+c-i-1)!}{(b-i)!(b+c-1)!}(b+c-i)\\
=&\frac{b!(b+c-1)!}{(2b+c)!}\bigg((b+c)\sum_{i=0}^{b}{2b+c-1-i\choose b+c-1}{i\choose 0}-\sum_{i=0}^{b}{2b+c-1-i\choose b+c-1}{i\choose 1}\bigg)\\
=&\frac{b!(b+c-1)!}{(2b+c)!}\bigg((b+c){2b+c\choose b+c}-{2b+c\choose b+c+1}\bigg)\\
=&\frac{b!(b+c-1)!}{(2b+c)!}\bigg((b+c)\frac{(2b+c)!}{(b+c)!b!}-\frac{(2b+c)!}{(b+c+1)!(b-1)!}\bigg)\\
=&1-\frac{b}{(b+c+1)(b+c)}
\end{align*}
By a similar process, the other double sum becomes
\[
\sum_{i=0}^{b} \sum_{j=0}^{c} \frac{{c \choose j}{b \choose i}{(i+j)!}{(a+b-i)}{(a+2b+c-i-j-1)!}}{(a+2b+c)!} = 1-\frac{b}{(b+a+1)(b+a)}.
\]
Now replacing the double sums by these simplified expressions we have
\begin{multline*}
P(GS_{a,b,c},2){=}\bigg(1-\frac{b}{(b+c+1)(b+c)}\bigg)-\frac{b+c}{a+2b+c} +\bigg(1-\frac{b}{(b+a+1)(b+a)}\bigg)-\frac{a+b}{a+2b+c}\\
=1-\frac{b}{(b+c+1)(b+c)}-\frac{b}{(b+a+1)(b+a)}.
\end{multline*}
Finally we note
\[
P(GS_{a,b,c},1)=1-P(GS_{a,b,c},2)=\frac{b}{(b+c+1)(b+c)}+\frac{b}{(b+a+1)(b+a)},
\]
establishing the result for $GS_{a,b,c}$.

The result for $P(GS_{a,b,c}^{+},2)$ follows by a similar argument, the only difference being the additional edge which \emph{cannot} be used in order to result in two trees.  So \eqref{eq:GS} would now become
\begin{multline*}
P(GS_{a,b,c}^{+},2)=\sum_{i=0}^{b} \sum_{j=0}^{a} \frac {{a \choose j}{b \choose i}{(i+j)!}{(b+c-i)}{(a+2b+c-i-j)!}}{(a+2b+c+1)!}-\frac{b+c}{a+2b+c+1} \\
+\sum_{i=0}^{b} \sum_{j=0}^{c} \frac{{c \choose j}{b \choose i}{(i+j)!}{(a+b-i)}{(a+2b+c-i-j)!}}{(a+2b+c+1)!} -\frac{a+b}{a+2b+c+1}.
\end{multline*}
The rest of the argument works in the same way as before.
\end{proof}

\begin{theorem}
We have the following probabilities.
\begin{align*}
P(\text{\emph{Paw}}_a,1)&=\frac16-\frac1{a+3}+\frac1{a+1}\\
P(\text{\emph{Di}}_a,1)&=\frac3{10}-\frac2{a+4}+\frac2{a+2}\\
P((K_4)_a,1)&=\frac25-\frac3{a+5}+\frac3{a+3}
\end{align*}
\end{theorem}
\begin{proof}
We will again compute the probability that there are two trees in the process.  However, in these cases there are many more possibilities to consider.  To simplify the situation, we make the following observation:  Every edge which is a leaf in the original graph will always be kept in the forest building process.  This indicates if there are multiple leaves off of a single vertex $v$, then we only need to know when the \emph{first} leaf was chosen.  This is because, by the first leaf, $v$ will have been seen by some edge.

So we now represent the remaining graphs from Figure~\ref{fig:families}, as shown in Figure~\ref{fig:condensed}, where $a$ corresponds to all of the $a$ leaves condensed down; and the remaining edges are labeled as indicated with each label other than $a$ corresponding to a single edge.

\begin{figure}[h!]
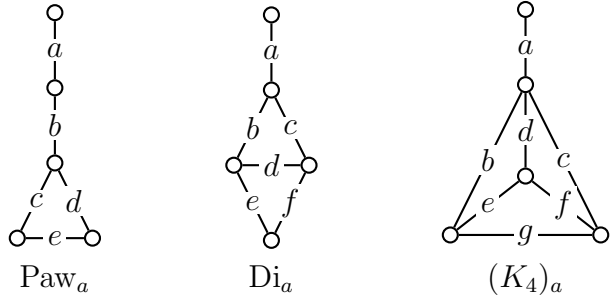

\centering

\picCONDENSED

\caption{The remaining three graphs with the leaves collapsed to a single edge $a$.}
\label{fig:condensed}
\end{figure}

For each graph, we now look at all possible ways to start selecting edges and end with a pair of disjoint edges.  We also find the probability of starting our selection in a particular way.  Recall that an edge marked $a$ corresponds to $a$ different edges, and so until we select that edge, we assume all $a$ of them haven't been seen and are available for picking; after selection by the observation, we can assume they have all been seen. (In other words, it is only the relative ordering of the different types of edges that matter.)

For the paw graph, we have the possibilities shown in Table~\ref{tab:paw} (the first column indicates every possible sequence of choices of edges until two trees are formed, while the second column indicates the probability of any one of those sequence of choices being made).  If we now sum all of these probabilities together, we get 
\[
P(\text{Paw}_a,2)=\frac56+\frac1{a+3}-\frac1{a+1},
\]
establishing the result (recall that $P(\text{Paw}_a,1)+P(\text{Paw}_a,2)=1$).

\begin{table}
\centering
\begin{tabular}{|l|l|}\hline
Start of edge orderings resulting in $2$ trees& Probabilities of an ordering \\ \hline\hline
$ac$, $ad$, $ae$ & $\displaystyle\vphantom{\bigg|}\frac{a}{a+4}\cdot\frac14$\\\hline
$be$, $eb$ & $\displaystyle\vphantom{\bigg|}\frac{1}{a+4}\cdot\frac{1}{a+3}$\\\hline
$ca$, $da$, $ea$ & $\displaystyle\vphantom{\bigg|}\frac{1}{a+4}\cdot\frac{a}{a+3}$\\\hline
$abe$& $\displaystyle\vphantom{\bigg|}\frac{a}{a+4}\cdot\frac14\cdot\frac13$ \\\hline
$bae$ & $\displaystyle\vphantom{\bigg|}\frac{1}{a+4}\cdot\frac{a}{a+3}\cdot\frac13$\\\hline
$cda$, $cea$, $dca$, $dea$, $eda$, $eca$ & $\displaystyle\vphantom{\bigg|}\frac{1}{a+4}\cdot\frac{1}{a+3}\cdot\frac{a}{a+2}$\\\hline
$ceda$, $cdea$, $dcea$, $deca$, $ecda$, $edca$ & $\displaystyle\vphantom{\bigg|}\frac{1}{a+4}\cdot\frac{1}{a+3}\cdot\frac{1}{a+2}\cdot\frac{a}{a+1}$\\\hline
\end{tabular}
\caption{Probabilities associated with $\text{Paw}_a$.}
\label{tab:paw}
\end{table}

The results for the remaining two graphs are established in the same way and the corresponding probabilities are given in Tables~\ref{tab:di} and \ref{tab:K4}.
\begin{table}
\centering
\begin{tabular}{|l|l|}\hline
Start of edge orderings resulting in $2$ trees& Probabilities of an ordering \\ \hline\hline
$ad$, $ae$, $af$ & $\displaystyle\vphantom{\bigg|}\frac{a}{a+5}\cdot\frac{1}{5}$\\\hline
$bf$, $ce$, $ec$, $fb$ & $\displaystyle\vphantom{\bigg|}\frac{1}{a+5}\cdot\frac{1}{a+4}$\\\hline
$da$, $ea$, $fa$ & $\displaystyle\vphantom{\bigg|}\frac{1}{a+5}\cdot\frac{a}{a+4}$\\\hline
$abf$, $ace$ & $\displaystyle\vphantom{\bigg|}\frac{a}{a+5}\cdot\frac{1}{5}\cdot\frac{1}{4}$\\\hline
$baf$, $cae$ & $\displaystyle\vphantom{\bigg|}\frac{1}{a+5}\cdot\frac{a}{a+4}\cdot\frac{1}{4}$\\\hline
$dea$, $dfa$, $eda$, $efa$, $fda$, $fea$ & $\displaystyle\vphantom{\bigg|}\frac{1}{a+5}\cdot\frac{1}{a+4}\cdot\frac{a}{a+3}$\\\hline
$defa$, $dfea$, $edfa$, $efda$, $fdea$, $feda$ & $\displaystyle\vphantom{\bigg|}\frac{1}{a+5}\cdot\frac{1}{a+4}\cdot\frac{1}{a+3}\cdot\frac{a}{a+2}$\\ \hline
\end{tabular}
\caption{Probabilities associated with $\text{Di}_a$.}
\label{tab:di}
\end{table}
\begin{table}
\centering
\begin{tabular}{|l|l|}\hline
Start of edge orderings resulting in $2$ trees& Probabilities of an ordering \\ \hline\hline
$ae$, $af$, $ag$ & $\displaystyle\vphantom{\bigg|}\frac{a}{a+6}\cdot\frac{1}{6}$\\ \hline
$bf$, $ce$, $dg$, $ec$, $fb$, $gd$ & $\displaystyle\vphantom{\bigg|}\frac{1}{a+6}\cdot\frac{1}{a+5}$\\ \hline
$ea$, $fa$, $ga$ & $\displaystyle\vphantom{\bigg|}\frac{1}{a+6}\cdot\frac{a}{a+5}$\\ \hline
$abf$, $adg$, $ace$ & $\displaystyle\vphantom{\bigg|}\frac{a}{a+6}\cdot\frac{1}{6}\cdot\frac{1}{5}$\\ \hline
$baf$, $dag$, $cae$ & $\displaystyle\vphantom{\bigg|}\frac{1}{a+6}\cdot\frac{a}{a+5}\cdot\frac{1}{5}$\\ \hline
$efa$, $ega$, $fea$, $fga$, $gea$, $gfa$ & $\displaystyle\vphantom{\bigg|}\frac{1}{a+6}\cdot\frac{1}{a+5}\cdot\frac{a}{a+4}$\\ \hline
$efga$, $egfa$, $fega$, $fgea$, $gefa$, $gfea$ & $\displaystyle\vphantom{\bigg|}\frac{1}{a+6}\cdot\frac{1}{a+5}\cdot\frac{1}{a+4}\cdot\frac{a}{a+3}$\\ \hline
\end{tabular}
\caption{Probabilities associated with $(K_4)_a$.}
\label{tab:K4}
\end{table}
\end{proof}

\section{Examples of graphs with the same probabilities}

Using the formulas from the theorems in the preceding section, we can now compute the probabilities for a large number of these graphs efficiently.  In particular,  we examined all graphs up through five hundred vertices in these families, and discovered several examples of families of non-isomorphic graphs which produce the same probabilities.

\begin{proposition}\label{prop:famA}
Given $s,t\ge 1$ with $s$ dividing into $2t(t+1)$, let $r=\frac{2t(t+1)}{s}$.  Then we have for all $k$
\[
P(GS_{r+3t+1,s,t},k)=
P(GS_{t,r+s+2t+1,t},k)=
P(GS_{3t+s+1,r,t},k).
\]
\end{proposition}

This immediately follows by applying the formulas for the probabilities from Theorem~\ref{thm:GS}.

\begin{figure}[h!]
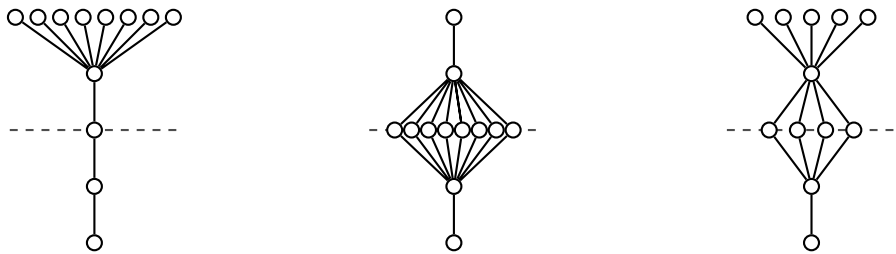

\centering

\MisaA

\caption{Examples of a set of three graphs with the same probabilities from Proposition~\ref{prop:famA}.  In this case $s=t=1$.}
\label{fig:sameP}
\end{figure}

\begin{proposition}\label{prop:famB}
Given $t\ge 1$, then we have for all $k$
\[
P(GS_{5t+3,t,2t},k)=
P(GS_{5t+1,t+1,2t+1},k).
\]
\end{proposition}

This also immediately follows by applying the formulas for probabilities from Theorem~\ref{thm:GS}.  We note that there were many other examples of pairs of glued star graphs which are not explained by Propositions~\ref{prop:famA} and \ref{prop:famB}.  A complete characterization of all such pairs of glued stars remains elusive.

Looking beyond glued stars, we found very few pairs of graphs with the same probabilities and the results do not seem to fit any patterns.  As an example, all pairs of graphs from the  $GS_{a,b,c}^{+}$ family up through $500$ vertices with the same probabilities are listed below (it is possible that this is a complete list for this family).
\begin{gather*}
P(GS_{17, 3, 9}^{+},k)= P(GS_{10, 9, 10}^{+},k)\\
P(GS_{28, 5, 9}^{+},k)= P(GS_{26, 8, 8}^{+},k)\\
P(GS_{103, 15, 48}^{+},k)= P(GS_{63, 71, 32}^{+},k)\\
P(GS_{95, 23, 53}^{+},k)= P(GS_{53, 66, 52}^{+},k)
\end{gather*}

\section{Conclusion}

We found the probabilities for all connected graphs which can form at most two trees in this forest building process.  A natural next step is to consider graphs with at most three trees.  As an example, two pairs of graphs with at most three trees and matching probabilities are given in Figure~\ref{fig:three}.  This is suggestive that these are the start of an infinite family of such graphs, but we have not yet established this.  One difficulty is that unlike the situation for two trees where only one probability needed to be computed (since the probabilities sum to one), this requires that two probabilities be computed.

\begin{figure}[h!]
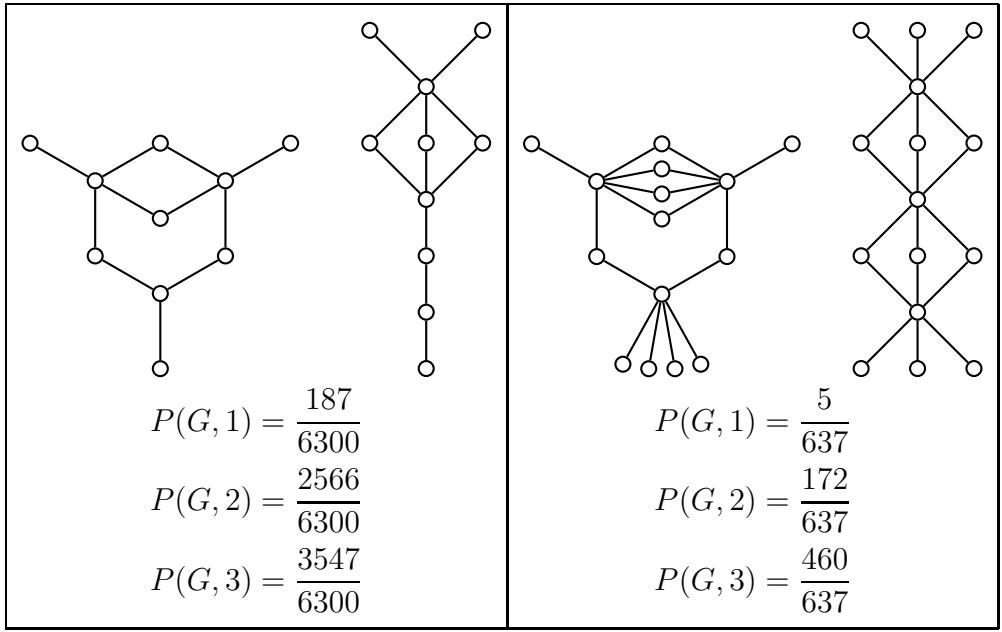

\centering

\begin{tabular}{|c|c|}\hline~&~\\[-8pt]
\FutureA~~~~~\FutureB&\FutureC~~~~~\FutureD\\ 
$\displaystyle\begin{array}{c}
\displaystyle P(G,1)=\frac{187}{6300}\\[10pt]
\displaystyle P(G,2)=\frac{2566}{6300}\\[10pt]
\displaystyle P(G,3)=\frac{3547}{6300}\\[10pt]\end{array}$&
$\displaystyle\begin{array}{c}
\displaystyle P(G,1)=\frac{5}{637}\\[10pt]
\displaystyle P(G,2)=\frac{172}{637}\\[10pt]
\displaystyle P(G,3)=\frac{460}{637}\\[10pt]\end{array}$\\ \hline
\end{tabular}

\caption{Two pairs of graphs with at most three trees and producing the same probabilities.}
\label{fig:three}
\end{figure}

\end{document}